\documentclass[twoside,11pt]{article}

\usepackage{blindtext}

\usepackage[preprint]{jmlr2e}

\hypersetup{breaklinks=true}
\usepackage{amsfonts}
\usepackage{amsmath}
\usepackage{bm}
\usepackage{nicefrac}
\usepackage{thmtools} 
\usepackage{thm-restate}
\usepackage{breakcites}
\usepackage[colorinlistoftodos]{todonotes}
\usepackage{dsfont}
\usepackage{comment}

\newtheorem{assumption}{Assumption}

\newcommand{\secref}[1]{Sec.~\ref{#1}}

\renewcommand{\eqref}[1]{Eq.~(\ref{#1})}
\newcommand{\defref}[1]{Def.~(\ref{#1})}
\newcommand{\lemref}[1]{Lemma~\ref{#1}}
\newcommand{\corref}[1]{Corollary~\ref{#1}}
\newcommand{\thmref}[1]{Thm.~\ref{#1}}
\newcommand{\propref}[1]{Proposition~\ref{#1}}
\newcommand{\appref}[1]{Appendix~\ref{#1}}
\newcommand{\assref}[1]{Assumption ~\ref{#1}}

\newcommand{\R}{\mathbb{R}}
\newcommand{\C}{\mathbb{C}}
\newcommand{\N}{\mathbb{N}}

\newcommand{\zero}{\boldsymbol{0}}

\newcommand{\abs}[1]{\left| #1 \right|}

\newcommand{\vertiii}[1]{{\left\vert\kern-0.25ex\left\vert\kern-0.25ex\left\vert #1 
    \right\vert\kern-0.25ex\right\vert\kern-0.25ex\right\vert}}

\renewcommand{\P}{\mathbb{P}}

\usepackage{xcolor}

\newcommand{\beq}{\begin{eqnarray*}}
\newcommand{\eeq}{\end{eqnarray*}}
\newcommand{\beqn}{\begin{eqnarray}}
\newcommand{\eeqn}{\end{eqnarray}}

\usepackage{ifmtarg}
\usepackage{xifthen}%
\newcommand{\ent}[1][]{%
\ifthenelse{\isempty{#1}}{%
\mathrm{H}
}{
\mathrm{H}^{(#1)}
}}

\newcommand{\loch}[1][]{%
\ifthenelse{\isempty{#1}}{%
\mathrm{h}
}{
\mathrm{h}^{(#1)}
}}

\newcommand{\Ncal}{\mathcal{N}}

\newcommand{\bigo}{\mathcal{O}}

\newcommand{\E}{\mathbb{E}}

\newcommand{\Sphere}{\mathbb{S}}

\newcommand{\bx}{\mathbf{x}}
\newcommand{\x}{\mathbf{x}}
\newcommand{\bw}{\mathbf{w}}

\newcommand{\bu}{\mathbf{u}}
\newcommand{\bv}{\mathbf{v}}
\newcommand{\bz}{\mathbf{z}}

\newcommand{\norm}[1]{\left\|#1\right\|}

\usepackage{lastpage}
\jmlrheading{23}{2024}{1-\pageref{LastPage}}{10/24; Revised 10/24}{10/24}{21-0000}{Daniel Barzilai and Ohad Shamir}

\ShortHeadings{Relative Deviation Bounds for Covariance Matrices}{Barzilai and Shamir}
\firstpageno{1}

\begin{document}

\title{Simple Relative Deviation Bounds\\ for Covariance and Gram Matrices}
\author{\name Daniel Barzilai 
\email daniel.barzilai@weizmann.ac.il \\
\name Ohad Shamir 
\email ohad.shamir@weizmann.ac.il \\    
Weizmann Institute of Science
}

\editor{}

\maketitle

\begin{abstract}
We provide non-asymptotic, \emph{relative} deviation bounds for the eigenvalues of empirical covariance and Gram matrices in general settings. Unlike typical uniform bounds, which may fail to capture the behavior of smaller eigenvalues, our results provide sharper control across the spectrum. Our analysis is based on a general-purpose theorem that allows one to convert existing uniform bounds into relative ones. The theorems and techniques emphasize simplicity and should be applicable across various settings.
\end{abstract}


\section{Introduction}

Many results in machine learning, statistics and other areas require controlling the eigenvalues of empirical covariance/Gram matrices. The goal of this paper is to provide \emph{non-asymptotic}, \emph{relative} deviation bounds, with an emphasis on generality and ease of use. By that, we mean that for random vectors $\bx_1,\ldots,\bx_n\in\R^d$, denoting $\hat{\Sigma}:=\frac{1}{n}\sum_{i=1}^n\bx_i \bx_i^\top\in\R^{d\times d}$ and $\Sigma:=\E\left[\hat{\Sigma}\right]\in\R^{d\times d}$, the bounds in this paper will be of the form 
\begin{align}\label{eq:example_statement}
\abs{\lambda_i\left(\hat\Sigma\right) - \lambda_i(\Sigma)} \leq C\cdot \lambda_i(\Sigma) \cdot \epsilon(n,d),
\end{align}
where $\epsilon(n,d)>0$ should be small, $C>0$ is some absolute constant, and $\lambda_i(\cdot)$ denotes the i'th largest eigenvalue of a matrix (where $\lambda_1\geq\lambda_2\geq\ldots$). There are, of course, mild conditions on $\bx_i$ which will be specified in the subsequent subsection.

This deviates from the typical bounds on $\lambda_i\left(\hat{\Sigma}\right)$ that are usually either \emph{uniform} \citep{rudelson1999random, vershynin2010introduction, adamczak2011sharp,tropp2012user,bunea2015sample, koltchinskii2017concentration, bandeira2023matrix, puchkin2023sharper, zhivotovskiy2024dimension, nakakita2024dimension} or \emph{asymptotic} \citep{marchenko1967distribution, baik2006eigenvalues, bai2008limit, feldheim2010universality, dornemann2023clt, atanasov2024scaling}. 
Uniform bounds typically control the spectral norm $\norm{\Sigma - \hat{\Sigma}}_2$ or the Frobenious norm $\norm{\Sigma - \hat{\Sigma}}_F$. 
These may be tight in bounding the largest eigenvalues of $\hat \Sigma$, but loose or even vacuous in bounding the smaller eigenvalues, especially when the spectral gap is large. 
For example, consider a case where $n, d$ are both large, and $\lambda_i\left(\Sigma\right) \lesssim \exp(-i)$. 
A uniform bound such as $\norm{\Sigma - \hat{\Sigma}}_2 \lesssim \norm{\Sigma}_2\sqrt{\frac{d}{n}}$ only tells us (via Weyl's inequality) that for every $i$, $\abs{\lambda_i\left(\Sigma\right) - \lambda_i\left(\hat\Sigma\right)} \lesssim \norm{\Sigma}_2\sqrt{\frac{d}{n}}$. But for most $i$ it holds that $\lambda_i\left(\Sigma\right)\ll \norm{\Sigma}_2\sqrt{\frac{d}{n}}$, so the uniform bound only ensures $\abs{\lambda_i\left(\hat\Sigma\right)}\lesssim \norm{\Sigma}_2\sqrt{\frac{d}{n}}$. This bound is, therefore, very loose when compared to a bound as in \eqref{eq:example_statement}. In particular, in such cases, uniform bounds cannot provide non-zero lower bounds for the smallest eigenvalues of $\hat{\Sigma}$, which is important for many applications.

In contrast to our bounds, asymptotic bounds characterize the limit distribution of the eigenvalues of $\hat\Sigma$ in the $n,d\to\infty$ limit when $\frac{d}{n}\to \gamma$ for some $\gamma\in(0,\infty)$. Unfortunately, it is generally difficult to convert such bounds into high-probability guarantees when $n$ and $d$ are finite \citep{vershynin2010introduction}. Furthermore, the convergence rate to the limit distribution may be relatively slow and depend on $\gamma$. Finally, the resulting bound is typically uniform and suffers from the same issues as mentioned before \citep{bai2010spectral}. Compared with these, our non-asymptotic bounds may be more precise for finite $n$ and $d$, do not require a fixed ratio between $n$ and $d$, are simpler, and should generally hold under weaker assumptions. The price we pay is that our bounds may be less precise in the limit when $\frac{d}{n}\to \gamma\in(0,\infty)$ due to the multiplicative constant $C>0$ in \eqref{eq:example_statement}. We therefore view these works as complementary.

For these reasons, relative and non-asymptotic bounds are critical in some applications and have therefore attracted attention in the literature by a series of excellent papers \citep{ipsen1998relative, ipsen2009refined, mas2015high, jirak2018relative, jirak2020perturbation, oliveira2016lower, ostrovskii2019affine, hucker2023note}. Many existing bounds have either required unnatural assumptions that are often not satisfied, primarily a large spectral gap (i.e lower bounds on $\max_{j\neq i}\abs{\lambda_i\left(\Sigma\right) - \lambda_j\left(\Sigma\right)}$). In contrast, our bounds make no assumptions on the eigenvalues $\lambda_i\left(\Sigma\right)$. Perhaps the most related results are those of \cite{barzilaigeneralization}, who developed relative bounds suited for distributions and applications that are specific to their analysis of high-dimensional kernel regression. This paper addresses more general and natural distributions, along with broader settings. \citet{oliveira2016lower} make just a mild fourth-moment assumption, but only provide a lower bound on $\lambda_i\left(\hat \Sigma\right)$ in the $d\leq n$ case. \citet{ostrovskii2019affine} provide bounds for $d\leq n$ and sub-gaussian distributions. They also provide bounds for heavy-tailed distributions when using a different estimator than the standard empirical covariance matrix. In contrast, our technique allows us to provide bounds for the empirical covariance matrix under mild distributional assumptions, as well as in high-dimensional settings ($d\geq n$).

Lastly, one of the main advantages of our bounds is simplicity, both of the bounds themselves and the techniques used. This simplicity does not come at the cost of tightness, as many of the bounds will be sharp up to multiplicative factors. The presentation in this paper assumes no specialized prior knowledge and should (hopefully) be generally accessible. 

\subsection{Reduction to Isotropic Random Vectors}\label{sec:isotropic}
Most theorems in this paper will consider the following standard setting:
\begin{assumption}\label{ass1}
     Let $X\in\R^{n\times d}$ be a matrix whose rows $\bx_1^\top,\ldots,\bx_n^\top$ are i.i.d. random vectors. Let
     $\Sigma:=\E\left[\frac{1}{n}X^\top X\right]\in\R^{d\times d}$, and assume that $\Sigma$ is invertible. Finally, let $\hat{\Sigma}:=\frac{1}{n}X^\top X\in\R^{d\times d}$.
\end{assumption}

We will often let $Z:=X\Sigma^{-\nicefrac{1}{2}}\in\R^{n\times d}$ and we note that \assref{ass1} implies that the rows $\bz_i$ of $Z$ are independent, \emph{isotropic} random vectors in $\R^d$, in the following sense:
\begin{definition}\label{def:isotropic}
    A random vector $\bz_i\in\R^d$ is said to be \emph{isotropic} if $\E\left[\bz_i\bz_i^\top\right]=I_d$. This is equivalent to saying that for any $v\in\R^d$, $\E[\langle \bz_i, v \rangle^2] = \norm{v}^2$.
\end{definition}

Indeed, it is straightforward to verify that $\bz_i=\Sigma^{-\nicefrac{1}{2}}\bx_i$ are isotropic, since $\E[\bz_i\bz_i^\top]=\E[\Sigma^{-\nicefrac{1}{2}}\bx_i\bx_i^\top \Sigma^{-\nicefrac{1}{2}}]=I_d$. Independence of $\bz_i$ follows directly from independence of $\bx_i$. In the subsequent subsections, we will reduce the task of providing relative deviation bounds as in \eqref{eq:example_statement} to uniform bounds on independent random vectors that are isotropic (i.e $\bz_i$).

\begin{remark}
    \assref{ass1} is slightly stronger than what we actually need for the subsequent theorems in the paper. In fact, it would suffice to assume that there exists some $Z:=[\bz_1, \ldots, \bz_n]^\top\in\R^{n\times d}$ such that $X= Z\Sigma^{\nicefrac{1}{2}}$ and the rows $\bz_i$ of $Z$ are independent, isotropic random vectors in $\R^d$. In this scenario, $\Sigma$ is not required to be invertible, and $\bx_i$ are not required to be identically distributed. Nevertheless, we opt to use $\assref{ass1}$ for improved clarity. Our main tool, \thmref{thm:gen_bound} is stated without \assref{ass1}, and therefore, extensions to settings beyond \assref{ass1} can easily be made.
\end{remark}





\subsection{Setting and Preliminaries}
We will need to make some assumptions on $\bz_i:=\Sigma^{-\nicefrac{1}{2}}\bx_i$, upon which the strength of the results will of course depend. In particular, stronger results will be applicable when $\bz_i$ are \emph{sub-gaussian}.
\begin{definition}\label{def:sub_gaus}
    A random vector $\bz_i\in\R^d$ (or random variable when $d=1$) is said to be sub-gaussian if 
    \[
    \norm{\bz_i}_{\psi_2}:=\sup_{\bu:\norm{\bu}=1}
    \sup_{p\geq 1}\frac{1}{\sqrt{p}}\left(\E[\abs{\langle \bz_i, \bu\rangle}^p]^{\nicefrac{1}{p}}\right)<\infty.
    \]
\end{definition}
There are multiple equivalent ways to define sub-gaussian vectors. In particular, the above implies that for any $\bu$ with $\norm{\bu}=1$ and some constant $c>0$, $\E\left[\exp\left(c\frac{\langle \bz_i, \bu\rangle^2}{\norm{\langle \bz_i, \bu\rangle}_{\psi_2}^2}\right)\right]\leq e$ and for any $t\geq 0$, $\P\left(\abs{\langle \bz_i, \bu\rangle}\geq t\right)\leq \exp\left(1- c\frac{t^2}{\norm{\langle \bz_i, \bu\rangle}_{\psi_2}^2}\right)$ \citep{vershynin2010introduction}. Perhaps the two most prominent examples of sub-gaussian vectors are Gaussians and bounded random vectors, so all results stated for sub-gaussian vectors also hold for these cases.

We will also state results that do not require sub-gaussianity. Such results will require weaker conditions that will be made explicit in the relevant chapters. 

The results in this paper will be stated in terms of the eigenvalues of the empirical second-moment matrix,  $\lambda_i\left(\hat{\Sigma}\right)$, but clearly, these also naturally provide bounds for the Gram matrix $XX^\top\in \R^{n\times n}$, since for any matrix $X$, $\lambda_i\left(X^\top X\right)=\lambda_i\left(X X^\top\right)$ for all $i\leq \min(n,d)$. 

Our results are typically stated for real-valued vectors, but both the proof of \thmref{thm:gen_bound} as well as many of the results we rely on can naturally be extended to the complex numbers \citep{vershynin2010introduction}. Unless specified otherwise, $\norm{\cdot}=\norm{\cdot}_2$ will always denote the standard 2-norm for vectors, and the spectral norm (operator $2$-norm) for matrices. $I_n$ denotes the $n$-dimensional identity matrix. We use the standard big-O notation, and the $\tilde{\bigo}(\cdot)$ notation to hide additional logarithmic factors. For $n\in\N$, $[n]$ denotes the set $\{1,\ldots,n\}$.

\section{Main Results}
The main tool that will allow us to obtain relative deviation bounds is based on the following \propref{prop:ostrowski}, which can be viewed as a generalization of Ostrowski's theorem for non-square matrices. Interestingly, it is non-probabilistic and relies on linear algebra alone.

\begin{restatable}{proposition}{Ostrowski}\label{prop:ostrowski}
    Let $Z \in\C^{n\times d}$ for some $n,d\in\N$, and $0\preceq \Sigma\in\C^{d\times d}$ be p.s.d. Then for any $1\leq i\leq \min(n,d)$ it holds that
    \begin{align*}
        \lambda_{i+d-\min(n,d)}\left(\Sigma\right)\lambda_{\min(n,d)}\left(Z^{*}Z\right) \leq \lambda_i\left(\Sigma^{\nicefrac{1}{2}}Z^* Z \Sigma^{\nicefrac{1}{2}}\right) \leq \lambda_i(\Sigma)\lambda_1\left(Z^{*}Z\right).
    \end{align*}
\end{restatable}
The proposition is mostly built upon some manipulations of the Courant-Fischer Min-Max theorem due to \cite{dancis1986quantitative}, and a self-contained proof is deferred to \appref{ap:gen_bound}. Variants of this proposition appeared in \cite{braun2005spectral, barzilaigeneralization} in the context of kernel regression, as well as in \citep{ostrovskii2019affine} for obtaining relative deviation bounds with different estimators. The analogs of \propref{prop:ostrowski} in \cite{braun2005spectral, ostrovskii2019affine} are for $d < n$, and do not extend to nontrivial bounds when $d > n$. 

We will now bring \propref{prop:ostrowski} to a more convenient form yielding the following \thmref{thm:gen_bound}, which will serve as our main tool for proving relative deviation bounds in the remainder of the paper. We state the theorem for real-valued matrices for consistency with the remainder of the paper, but the same proof holds over $\C$.

\begin{restatable}{theorem}{GenBound}\label{thm:gen_bound}
Let $X, Z\in \R^{n\times d}$, and $\Sigma\in\R^{d\times d}$ be matrices such that $X= Z\Sigma^{\nicefrac{1}{2}}$ and $\hat{\Sigma}:=\frac{1}{n}X^\top X$.

\begin{enumerate}
    \item If $d\leq n$ then
    \begin{align*}
    \abs{\lambda_i\left(\hat\Sigma\right) - \lambda_i\left(\Sigma\right)} \leq \lambda_i\left(\Sigma\right)\norm{\frac{1}{n}Z^\top Z - I_d}_2.
    \end{align*}
    \item If $d \geq n$ then
    \begin{align*}
    \lambda_{i+d-n}\left(\Sigma\right)\left(1 - \norm{\frac{1}{d}ZZ^\top - I_n}_2\right) \leq \frac{n}{d}\cdot \lambda_i\left(\hat\Sigma\right) \leq \lambda_i\left(\Sigma\right)\left(1 + \norm{\frac{1}{d}ZZ^\top - I_n}_2\right).
    \end{align*}
\end{enumerate}
\end{restatable}
\begin{proof}
    Using Weyl's inequality, \citep{horn2012matrix}[Theorem 4.3.1] for any symmetric matrix $A$ it holds that
    \begin{align}\label{eq:weyl}
        1 - \norm{A - I}_2 \leq \lambda_i\left(A\right) \leq 1 + \norm{A - I}_2.
    \end{align}
    For $d\leq n$, \propref{prop:ostrowski} implies that
    \begin{align*}
    \lambda_{i}\left(\Sigma\right)\lambda_{d}\left(\frac{1}{n}Z^{\top}Z\right) \leq \lambda_i\left(\hat\Sigma\right) \leq \lambda_i(\Sigma)\lambda_1\left(\frac{1}{n}Z^{\top}Z\right).
    \end{align*}
    Bounding the eigenvalues of $\frac{1}{n}Z^{\top}Z$ using \eqref{eq:weyl} yields
    \begin{align*}
    \lambda_{i}\left(\Sigma\right)\left(1 - \norm{\frac{1}{n}Z^{\top}Z - I_d}_2\right)
    \leq \lambda_i\left(\hat\Sigma\right) 
    \leq \lambda_i(\Sigma)\left(1 + \norm{\frac{1}{n}Z^{\top}Z - I_d}_2\right).
    \end{align*}
    This is equivalent to what we needed to prove. 
    
    For the $d\geq n$ case, \propref{prop:ostrowski} combined with the fact that $\lambda_i\left(ZZ^\top\right)=\lambda_i\left(Z^\top Z\right)$ implies
    \begin{align*}
    \lambda_{i+d-n}\left(\Sigma\right)\lambda_{n}\left(\frac{1}{d}ZZ^{\top}\right) \leq \lambda_i\left(\frac{n}{d}\hat\Sigma\right) \leq \lambda_i(\Sigma)\lambda_1\left(\frac{1}{d}ZZ^{\top}\right).
    \end{align*}
    Again, the theorem follows by applying \eqref{eq:weyl} to $\frac{1}{d}ZZ^\top$.
\end{proof}

To see the utility of \thmref{thm:gen_bound}, consider the low-dimensional case ($d\leq n$) and rows $\bz_i$ of $Z$ that are independent, mean-zero isotropic random vectors. Then by \defref{def:isotropic}, their covariance matrix is $\E[\frac{1}{n}Z^{\top}Z]=I_d$, and one should thus expect $\norm{\frac{1}{n}Z^\top Z - I_d}_2$ to be small for sufficiently large $n$. Thus, \thmref{thm:gen_bound} reduces the task of deriving relative deviation bounds for $\hat \Sigma$ to the task of deriving uniform bounds $\norm{\frac{1}{n}Z^\top Z - I_d}_2$ for isotropic vectors. As mentioned in the introduction, uniform bounds for isotropic vectors have been the subject of many past works, and are generally well understood. The power of \thmref{thm:gen_bound} is allowing us to leverage these results to obtain relative bounds.

We note that the $\frac{n}{d}$ scaling in the bound of the high-dimensional case ($d\geq n$) is strictly necessary. This follows from the fact that $\hat\Sigma$ is scaled by $\frac{1}{n}$ and not $\frac{1}{d}$. Indeed, for $i\leq \min(n,d)$ it always holds that  $\lambda_i\left(\hat{\Sigma}\right)=\lambda_i\left(\frac{1}{n}X^\top X\right)=\frac{d}{n}\lambda_i\left(\frac{1}{d}XX^\top\right)$, so if, for example, the entries of $X$ are all i.i.d. with mean $0$ and variance $1$, one should expect $\frac{1}{d}XX^\top \approx I_n$. In this scenario, since $\Sigma=I_d$, we obtain $\lambda_i(\hat{\Sigma})\approx \frac{d}{n}\lambda_i(\Sigma)$.

\subsection{Low-Dimensional Case $(d \leq n)$}

In this section, we apply \thmref{thm:gen_bound} to obtain relative deviation bounds in the low-dimensional case, when $d\leq n$. The high dimensional case of $d\geq n$ will be treated in the following section. As mentioned in the previous section, thanks to \thmref{thm:gen_bound} it only remains to bound $\norm{\frac{1}{n}Z^\top Z - I_d}_2$ where the rows of $Z$ are isotropic. The following result from \cite{vershynin2018high} treats the case where the rows $\bz_i$ of $Z$ are mean-zero sub-gaussian:

\begin{theorem}
[\citet{vershynin2018high} Theorem 4.6.1]\label{thm:low_dim_isotropic}
    Let $Z:=[\bz_1, \ldots, \bz_n]^\top\in\R^{n\times d}$ be a matrix whose rows $\bz_i$ are independent, mean-zero, sub-gaussian isotropic random vectors in $\R^d$ with $K:=\max_i\norm{\bz_i}_{\psi_2}$. Then for some absolute constant $C >0$ and any $t\geq 0$ it holds w.p. at least $1-2\exp(-t^2)$
    \begin{align*}
        \norm{\frac{1}{n}Z^\top Z - I_d}_2\leq CK^2\max(\epsilon, \epsilon^2) \qquad\text{where}\qquad \epsilon:=\sqrt{\frac{d}{n}}+\frac{t}{\sqrt{n}}.
    \end{align*}
\end{theorem}

Combined with \thmref{thm:gen_bound}, we immediately obtain the following relative deviation bounds for $\hat \Sigma$:

\begin{restatable}[Low-Dimensional, Sub-Gaussian]{theorem}{SubGaussian}\label{thm:low_dim_sub_gaus}
    Under \assref{ass1} with $d\leq n$, assume further that $\bx_i$ are mean-zero, and that $\bz_i:=\Sigma^{-\nicefrac{1}{2}}\bx_i$ are sub-gaussian with $K:=\max_i\norm{\bz_i}_{\psi_2}$. Then for some absolute constant $C >0$ and any $t\geq 0$ it holds w.p. at least $1-2\exp(-t^2)$ that for all $i\in[d]$,
    \begin{align*}
        \quad \abs{\lambda_i\left(\hat\Sigma\right)- \lambda_i(\Sigma)} \leq CK^2 \lambda_i(\Sigma)\max\left(\epsilon, \epsilon^2\right) 
        \qquad \text{where}\qquad\epsilon:=\sqrt{\frac{d}{n}}+\frac{t}{\sqrt{n}}.
    \end{align*} 
\end{restatable}

\begin{proof}
Let $Z:=X\Sigma^{-\nicefrac{1}{2}}\in\R^{n\times d}$  so that $Z:=[\bz_1, \ldots, \bz_n]^\top$. \thmref{thm:gen_bound} gives
\begin{align}\label{eq:proof1}
    \abs{\lambda_i\left(\hat\Sigma\right) - \lambda_i\left(\Sigma\right)} \leq \lambda_i\left(\Sigma\right)\norm{\frac{1}{n}Z^\top Z - I_d}_2.
\end{align}
As described in \secref{sec:isotropic}, $\bz_i$ are independent and isotropic. Furthermore, $\bz_i$ are also mean-zero as $\E[\bz_i]=\Sigma^{-\nicefrac{1}{2}}\E[\bx_i]=\zero$. So the conditions of \thmref{thm:low_dim_isotropic} hold, and applying this theorem to bound $\norm{\frac{1}{n}Z^\top Z - I_d}_2$ in \eqref{eq:proof1} completes the proof.
\end{proof}

Thus, $n=\bigo (K^4 d)$ samples suffice to obtain good relative deviation bounds. We note that if one needs only a bound on the \emph{largest} eigenvalues of $\hat\Sigma$, uniform deviation bounds may provide a better dependence on $d$ using some notion of an intrinsic dimension \citep{zhivotovskiy2024dimension}. \citet{ostrovskii2019affine} incorporated a notion of degrees of freedom in a relative bound, but nevertheless, for most eigenvalues in the spectrum, \thmref{thm:low_dim_sub_gaus} improves upon \citet{ostrovskii2019affine}[Eq. 12] by a $\log(d)$ factor. \thmref{thm:low_dim_sub_gaus} also improve upon \citet{hucker2023note}[Corollaries 2,3] who showed $\frac{\lambda_i(\Sigma)}{2} \leq \lambda_i\left(\hat{\Sigma}\right) \leq 2\lambda_i(\Sigma)$. Interestingly, the bounds for the isotropic case given by \thmref{thm:low_dim_isotropic} cannot be strengthened by more than a multiplicative constant, even if we assume that all entries of $\bz_i:=\Sigma^{-\nicefrac{1}{2}}\bx_i$ are i.i.d. Consider the special case where $\bx_i \sim \Ncal\left(0, \Sigma\right)$ for some invertible $\Sigma$ (namely, a zero-mean Gaussian distribution with covariance matrix $\Sigma$). It is well known that this implies that the entries of $\bz_i$ are i.i.d. standard Gaussian random variables, for which bounds on the singular values of $Z:=[\bz_1, \ldots, \bz_n]^\top$ are well known \citep{davidson2001local, vershynin2010introduction}. We thus have the following:

\begin{restatable}[Gaussian Entries]{theorem}{Gaussian}\label{thm:gaus}
    Consider the special case of \assref{ass1} with $d\leq n$ where $\bx_i \sim \Ncal\left(0, \Sigma\right)$. Then for any $t\geq 0$ it holds w.p. at least $1-2\exp(-\frac{t^2}{2})$ that for all $i\in[d]$,
    \begin{align*}
    \abs{\lambda_i\left(\hat\Sigma\right)- \lambda_i(\Sigma)} \leq \lambda_i(\Sigma)\left(2\epsilon + \epsilon^2\right) \qquad \text{where}\qquad\epsilon:=\sqrt{\frac{d}{n}}+\frac{t}{\sqrt{n}}.
    \end{align*} 
\end{restatable}
\begin{proof}
    Let $Z:=X\Sigma^{-\nicefrac{1}{2}}\in\R^{n\times d}$, implying that the entries of $Z$ are i.i.d. standard Gaussians $\Ncal(0,1)$. By \citep{vershynin2010introduction}[Corollary 5.35] it holds with probability at least $1-2\exp(-\frac{t^2}{2})$ that for all $i\in[n]$, 
\begin{align}
    \left(1-\epsilon\right)^2\leq \lambda_i\left(\frac{1}{n}Z^\top Z\right) \leq \left(1+\epsilon\right)^2
    \qquad \text{where}\qquad\epsilon:=\sqrt{\frac{d}{n}}+\frac{t}{\sqrt{n}}.
\end{align}
Via Weyl's inequality the above implies that for all $i\in[n]$, $\lambda_i\left(\frac{1}{n}Z^\top Z - I_d\right)\leq 2\epsilon + \epsilon^2$, and thus $\norm{\frac{1}{n}Z^\top Z - I_d}_2\leq 2\epsilon + \epsilon^2$. Combining with \thmref{thm:gen_bound} concludes the proof.
\end{proof}

For the special case of Gaussian random vectors, the bounds in \thmref{thm:gaus} improve upon the bounds of \thmref{thm:low_dim_sub_gaus} in the sense that the constants are specified exactly. Nevertheless, the asymptotic dependence on the number of samples $n$ and the dimension $d$ remain the same.

\subsubsection{Bounds Without a Dependence on Sub-Gaussian Norm}
The dependence on $K$ in \thmref{thm:low_dim_sub_gaus} may be undesirable when the sub-gaussian norm is large relative to $d$. Consider for example the case when $\bz_i$ is distributed uniformly in the set $\{\sqrt{d}e_i\}_{i=1}^d$, where $e_i$ denote the standard basis vectors. It is straightforward to verify that $\norm{\bz_i}_{\psi_2} \gtrsim \sqrt{\frac{d}{\log(d)}}$ (for example, consider taking in \defref{def:sub_gaus} $u=e_1$ and $p=\log(d)$). In such a case, the bound in \thmref{thm:low_dim_sub_gaus} will exhibit a very poor dependence on $d$. To fix this, we derive an analog of \thmref{thm:low_dim_sub_gaus}, which depends on $\norm{\bz_i}_2$ instead of $\norm{\bz}_{\psi_2}$.

\begin{theorem}[Low-Dimensional, Bounded Norm] Under \assref{ass1} with $d\leq n$, let $m>0$ be a number s.t. $\bz_i:=\Sigma^{-\nicefrac{1}{2}}\bx_i$ satisfy $\norm{\bz_i}_2 \leq \sqrt{m}$ a.s. for all $i\in[n]$. Then for some absolute constant $c>0$ and any $t\geq0$, it holds w.p. at least $1-2d\exp(-ct^2)$ that for all $i\in[d]$,
    \begin{align*}
        \abs{\lambda_i\left(\hat\Sigma\right) - \lambda_i(\Sigma)} \leq \lambda_i(\Sigma)\max\left(\epsilon, \epsilon^2\right) \qquad\text{where}\qquad \epsilon:=t\sqrt{\frac{m}{n}}.
    \end{align*}
\end{theorem}
The proof is analogous to that of \thmref{thm:low_dim_sub_gaus}, where the only difference is that $\norm{\frac{1}{n}Z^\top Z - I_d}_2$ is bounded using \cite{vershynin2010introduction}[Theorem 5.41] instead of \thmref{thm:low_dim_isotropic}. By the definition of isotropic vectors, $\E\left[\bz_i\right]=\sqrt{d}$ and as such, it always holds that $m\geq d$. Furthermore, in order for the theorem to hold with probability at least $1-\delta$ for some $\delta>0$, one would have to take $t\geq\sqrt{\frac{\log\left(\frac{2d}{\delta}\right)}{c}}$, introducing an additional $\log(d)$ factor. This means that in the heavy-tailed case, $n$ has to be on the order of $m\log(d)$, which is at least $d\log(d)$. This dependence on $d$ is weaker than the dependence needed in \thmref{thm:low_dim_sub_gaus} by a $\log(d)$ factor when $K=O(1)$, but is stronger when $K\gtrsim \log\left(d\right)^{\nicefrac{1}{4}}$. Nevertheless, this $\log(d)$ factor cannot be removed without further assumptions (see the discussion after the proof of Thm. 5.41 in \cite{vershynin2010introduction}).

\subsubsection{Variants and Extensions}

There are many other possible bounds on the eigenvalues of $\frac{1}{n}Z^\top Z$ that together with \thmref{thm:gen_bound} can yield relative deviation bounds beyond those presented in the previous section \citep{rudelson2010non}. 
For example,
\cite{koltchinskii2015bounding, yaskov2014lower, yaskov2015sharp} provide lower bounds on the smallest eigenvalues of $\frac{1}{n}Z^\top Z$ when the rows $\bz_i$ have finite $2+\eta$ moments (for any $\eta>0$). For $\eta > 2$, the bounds match those of \thmref{thm:low_dim_isotropic} up to constants which depend on the moments. Bounds that apply also to the largest eigenvalues under similar $2+\eta$ moment assumptions are given in \citet{mendelson2014singular, guedon2017interval, tikhomirov2018sample}. 

\subsection{High-Dimensional Case ($d \geq n$)}

We now derive analogs of the theorems in the previous section in the high-dimensional ($d\geq n$) case. In such a case, $\lambda_i\left(\hat{\Sigma}\right)=0$ for $i>n$ and we therefore concern ourselves only with the first $n$ eigenvalues. The simplest case is if the entries $z_{ij}$ of $Z:=X\Sigma^{-\nicefrac{1}{2}}$ are independent and mean-zero, unit-variance sub-gaussian, with $\norm{z_{ij}}_{\psi_2} \leq K$ for some $K>0$. In such a case, \thmref{thm:low_dim_isotropic} can be used with $Z^\top$ instead of $Z$, reversing the roles of $n$ and $d$.

\begin{theorem}[High-Dimensional, Sub-Gaussian Entries]\label{thm:high_dim_sub_gaus}
    Under \assref{ass1} with $d\geq n$, assume further that $\bx_i$ are mean-zero, and that the entries $z_{ij}$ of $\bz_i:=\Sigma^{-\nicefrac{1}{2}}\bx_i$ are independent, sub-gaussian random variables with variance $1$ and sub-gaussian norm $\norm{z_{ij}}_{\psi_2} \leq K$ for all $i\in[n],j\in[d]$. Then for some absolute constant $C >0$ and any $t\geq 0$ it holds w.p. at least $1-2\exp(-t^2)$ that for all $i\in[n]$,
    \begin{align*}
    \lambda_{i+d-n}\left(\Sigma\right)\left(1 - CK^2\max\left(\epsilon, \epsilon^2\right)\right) \leq \frac{n}{d}\cdot \lambda_i\left(\hat\Sigma\right) \leq \lambda_i\left(\Sigma\right)\left(1 + CK^2\max\left(\epsilon, \epsilon^2\right)\right),
    \end{align*}
    where $\epsilon:=\sqrt{\frac{n}{d}}+\frac{t}{\sqrt{d}}$.
\end{theorem}
\begin{proof}
Again, let $Z:=X\Sigma^{-\nicefrac{1}{2}}\in\R^{n\times d}$, then \thmref{thm:gen_bound} gives
\begin{align}\label{eq:proof2}
    \lambda_{i+d-n}\left(\Sigma\right)\left(1 - \norm{\frac{1}{d}ZZ^\top - I_n}_2\right) \leq \frac{n}{d}\cdot \lambda_i\left(\hat\Sigma\right) \leq \lambda_i\left(\Sigma\right)\left(1 + \norm{\frac{1}{d}ZZ^\top - I_n}_2\right).
\end{align}
Now let $\tilde{\bz}_1, \ldots, \tilde{\bz}_d$ denote the rows of $Z^\top$ (instead of $Z$) s.t. $Z=[\tilde{\bz}_1, \ldots, \tilde{\bz}_d]$. $\tilde \bz_i$ are independent by assumption, and mean-zero as $\E[Z]=\E[X]\Sigma^{-\nicefrac{1}{2}}=\zero$. Since the entries of $\tilde{\bz}_i$ are independent, mean-zero, and have variance $1$, $\tilde{\bz}_i$ have unit covariance and are therefore isotropic. Furthermore, as the entries $z_{ij}$ are sub-gaussian, by \cite{vershynin2010introduction}[Lemma 5.24], $\tilde{\bz}_i$ are sub-gaussian random vectors with $\norm{\tilde{\bz}_i}_{\psi_2} \leq \tilde CK$ for some constant $\tilde C>0$.

Thus, \thmref{thm:low_dim_isotropic} can be used with $Z^\top$ instead of $Z$ (where the roles of $n$ and $d$ are switched), and we obtain that with probability at least $1-2\exp(-t^2)$,
\begin{align}
    \norm{\frac{1}{d}Z Z^\top - I_n}_2\leq CK^2\max(\epsilon, \epsilon^2) \qquad\text{where}\qquad \epsilon:=\sqrt{\frac{n}{d}}+\frac{t}{\sqrt{d}}.  
\end{align}
This together with \eqref{eq:proof2} completes the proof.
\end{proof}

We note that as discussed after the proof of \thmref{thm:gen_bound}, the $\frac{n}{d}$ factor is not a weakness of our bounds, but rather a necessary re-scaling due to the fact that $\hat{\Sigma}$ is scaled by $\frac{1}{n}$ instead of $\frac{1}{d}$. We now move on to the case where the rows $\bz_i$ are independent, but not necessarily the entries $z_{ij}$ of $Z$. In this case, our results will require the assumption that $\norm{\bz_i}$ is constant, in which case it must equal $\sqrt{d}$ (as by \defref{def:isotropic}, $\E[\bz_i\bz_i^\top]=I_d$). This equality condition is, of course, more restrictive than the ones in the previous theorems. Nevertheless, it allows us to obtain the following results:

\begin{theorem}[High-Dimensional, Independent Vectors]\label{thm:high_dim}
        Under \assref{ass1} with $d\geq n$, assume further that $\bz_i:=\Sigma^{-\nicefrac{1}{2}}\bx_i$ satisfy $\norm{\bz_i}=\sqrt{d}$ a.s. for
        every $i\in [n]$. 
    \begin{enumerate}
        \item If $\bz_i$ are also sub-gaussian, then for some constants $C_K,c_K>0$ which depend only on the sub-gaussian norm $K=\max_i\norm{\bz_i}_{\psi_2}$, $\epsilon:=C_K\sqrt{\frac{n}{d}} + \frac{t}{\sqrt{d}}$, and any $t\geq 0$, it holds w.p. at least $1-2\exp(-c_Kt^2)$ that for all $i\in[n]$,
        \begin{align}\label{eq:high_dim1}
        \lambda_{i+d-n}\left(\Sigma\right)\left(1 - \max\left(\epsilon, \epsilon^2\right)\right) \leq \frac{n}{d}\cdot \lambda_i\left(\hat\Sigma\right) \leq \lambda_i\left(\Sigma\right)\left(1 + \max\left(\epsilon, \epsilon^2\right)\right).
        \end{align}
        \item For any $p\in\N$ let $K(p) := \max_{i\in[n]}\sup_{x\in\Sphere^{d-1}}\E_{\bz_i}[\abs{\langle \bz_i, x \rangle}^{p}]^{\frac{1}{p}}$, then for all $i\in[n]$,
        \begin{align}\label{eq:high_dim2}
        \lambda_{i+d-n}(\Sigma)\left(1- \epsilon\right) \leq \frac{n}{d}\E\left[\lambda_i\left(\hat\Sigma\right)\right] \leq \lambda_i(\Sigma)\left(1+ \epsilon \right).
        \end{align}
        where $\epsilon=\sqrt{\frac{B(n,p)}{d}}$ with
        \begin{align*}
        B(n, p) := C\frac{p}{\log(p+1)}n^{\frac{1}{p}}\max\left(n, n^\frac{1}{p}K(2p)^2\right) \log(n)
        \end{align*}
        for some absolute constant $C>0$. 
    \end{enumerate}
\end{theorem}

The proof of these bounds, as usual, involves bounding $\norm{\frac{1}{d}Z Z^\top - I_n}_2$ and then using \thmref{thm:gen_bound}. The first part of the theorem, \eqref{eq:high_dim1}, is relatively straightforward and handled by \cite{vershynin2010introduction}[Theorem 5.58]. The second part, \eqref{eq:high_dim2}, which depends on the moment bounds $K(p)$, is trickier and combines existing bounds with a hypercontractivity argument. The full proof is presented in \ref{ap:high_dim}. 

\subsection{Bounds for Square and Nearly Square Matrices $(d\approx n)$}\label{seq:square}
Even though the bounds of \thmref{thm:low_dim_sub_gaus} are sharp up to a multiplicative constant, such constant makes the bounds weaker when $\sqrt{\frac{d}{n}}$ is not small. In particular, the lower bound on the eigenvalues of $\hat{\Sigma}$ may be vacuous. Nevertheless, \cite{rudelson2009smallest} provide lower bounds for the singular values of $Z$ that are more suitable for square and nearly square matrices. 
\begin{proposition}[\cite{rudelson2009smallest} Theorem 1.1]\label{prop:iso_square}
    Let $Z \in\R^{n\times d}$ be an $n\times d$ matrix with $d\leq n$, whose entries $z_{ij}$ are i.i.d. mean-zero random variables with variance $1$ and sub-gaussian norm $\norm{z_{ij}}_{\psi_2} \leq K$. Then for some constants $C_K, c_K >0$ that depend (polynomially) only on $K$, and any $t\geq 0$, it holds w.p. at least $1-\left(\frac{C_K}{t}\right)^{n-d+1} + e^{-c_Kn}$ that
    \begin{align*}
        \lambda_d\left(\frac{1}{n}Z^\top Z\right) \geq \frac{1}{t^2} \left(1 - \sqrt{\frac{d-1}{n}}\right)^2
    \end{align*}
\end{proposition}
The following result follows from combining their bounds with \propref{prop:ostrowski}.

\begin{restatable}[Square/Nearly-Square Matrices]{theorem}{square}\label{thm:square}
    Under \assref{ass1}, assume further that $\bx_i$ are mean-zero, and that the entries $z_{ij}$ of $\bz_i:=\Sigma^{-\nicefrac{1}{2}}\bx_i$ are i.i.d. sub-gaussian random variables with variance $1$ and sub-gaussian norm $\norm{z_{ij}}_{\psi_2} \leq K$ for all $i\in[n],j\in[d]$. Then for some absolute constant $\tilde{C} >0$, constants $C_K, c_K >0$ that depend (polynomially) only on $K$, and any $t_1,t_2\geq 0$ the following hold:

    \begin{enumerate}
        \item If $n\geq d$, then w.p. at least $1-\left(\frac{C_k}{t_1}\right)^{n-d+1} + e^{-c_Kn}-2\exp(-t_2^2)$ that for all $i\in[d]$,
        \begin{align}\label{eq:square_p1}
        \lambda_{i}\left(\Sigma\right)\left(\frac{1}{t_1^2}\left(1 - \epsilon_1 \right)^2\right) 
        \leq \lambda_i\left(\hat\Sigma\right) \leq \lambda_i\left(\Sigma\right)\left(1 + \tilde{C}K^2\max\left(\epsilon_2, \epsilon_2^2\right)\right),
        \end{align}
        where $\epsilon_1:=\sqrt{\frac{d-1}{n}}$ and $\epsilon_2:=\sqrt{\frac{d}{n}}+\frac{t_2}{\sqrt{n}}$.

        \item If $d\geq n$, then w.p. at least $1-\left(\frac{C_k}{t_1}\right)^{d-n+1} + e^{-c_Kd}-2\exp(-t_2^2)$ that for all $i\in[n]$,
        \begin{align}\label{eq:square_p2}
        \lambda_{i+d-n}\left(\Sigma\right)\left(\frac{1}{t_1^2}\left(1 - \epsilon_1 \right)^2\right)  
        \leq \frac{n}{d} \lambda_i\left(\hat\Sigma\right) \leq  \lambda_i\left(\Sigma\right)\left(1 + \tilde{C}K^2\max\left(\epsilon_2, \epsilon_2^2\right)\right),
        \end{align}
        where $\epsilon_1:=\sqrt{\frac{n-1}{d}}$ and $\epsilon_2:=\sqrt{\frac{n}{d}}+\frac{t_2}{\sqrt{d}}$.
    \end{enumerate}
\end{restatable}
\begin{proof}
    The upper bounds for \eqref{eq:square_p1} and \eqref{eq:square_p2} are given by \thmref{thm:low_dim_sub_gaus} and \thmref{thm:high_dim_sub_gaus} respectively.

    Let $Z:=X\Sigma^{-\nicefrac{1}{2}}\in\R^{n\times d}$ so that $Z:=[\bz_1, \ldots, \bz_n]^\top$. It is readily seen that $z_{ij}$ are mean-zero, as $\E[Z]=\E[X]\Sigma^{-\nicefrac{1}{2}}=\zero$, and therefore satisfy all the conditions of \propref{prop:iso_square}. 
    To prove the lower bound of \eqref{eq:square_p1}, by \propref{prop:ostrowski} it holds that
    \begin{align*}
    \lambda_{i}\left(\Sigma\right)\lambda_{d}\left(\frac{1}{n}Z^{\top}Z\right) \leq \lambda_i\left(\hat\Sigma\right), 
    \end{align*}
    from which \eqref{eq:square_p1} follows by bounding $\lambda_{d}\left(\frac{1}{n}Z^{\top}Z\right)$ using \propref{prop:iso_square}.

    Analogously for the lower bound of \eqref{eq:square_p2}, by \propref{prop:ostrowski} it holds that
    \begin{align*}
        \lambda_{i+d-n}\left(\Sigma\right)\lambda_{n}\left(\frac{1}{d}ZZ^{\top}\right) \leq \frac{n}{d}\lambda_i\left(\hat\Sigma\right), 
    \end{align*}
    from which \eqref{eq:square_p2} follows by bounding $\lambda_{n}\left(\frac{1}{d}ZZ^{\top}\right)$ by applying  \propref{prop:iso_square} on $Z^\top$ instead of $Z$ (with the roles of $n$ and $d$ reversed).
\end{proof}

Unlike in the $d\gg n$ or $d \ll n$ cases, the bounds that one can expect in the $d\approx n$ regime are somewhat looser. However, this is not a limitation of our method and is expected by asymptotic results. Consider for example the case when $\bx_i$ are isotropic (so that $\bx_i=\bz_i$ and $\lambda_i(\Sigma)=1$) and $\frac{d}{n}\to\gamma$ for some fixed $\gamma \in (0,1)$. In this case, the Bai-Yin theorem \citep{bai2008limit} states that asymptotically all eigenvalues $\lambda_i\left(\hat\Sigma\right)$ will be in the range $\left[(1-\sqrt{\gamma})^2, (1+\sqrt{\gamma})^2\right]$. This matches the lower bound of \thmref{thm:square} up to a multiplicative factor given by $t_1^2$. The upper bound of \thmref{thm:square} roughly gives $1+\tilde{C}K^2 \sqrt{\gamma}$ closely resembling the asymptotic upper bound of $1+2\sqrt{\gamma}+\gamma$. The tightness of \thmref{thm:square} in the $d>n$ case is analogous but with the roles of $d$ and $n$ reversed. We also remark that for the square case, when $d=n$, \eqref{eq:square_p1} and \eqref{eq:square_p2} yield the same bound.


\acks{The authors thank Boaz Nadler and Ofer Zeitouni for helpful discussions during the initial stages of this manuscript. 
This research is supported in part by European Research Council (ERC) grant 754705, by the Israeli Council for Higher Education (CHE) via the Weizmann Data Science Research Center and by research grants from the Estate of Harry Schutzman and the Anita James Rosen Foundation. }

\appendix
\section{Omitted Proofs}
\subsection{Proof of \propref{prop:ostrowski}}\label{ap:gen_bound}
\Ostrowski*
\begin{proof}
    First, we note that as $\lambda_i\left(\Sigma^{\nicefrac{1}{2}}Z^* Z \Sigma^{\nicefrac{1}{2}}\right) = \lambda_i\left(Z\Sigma Z^*\right)$ for all $i\leq \min(n,d)$, we equivalently prove that
    \begin{align*}
        \lambda_{i+d-\min(n,d)}\left(\Sigma\right)\lambda_{\min(n,d)}\left(Z^{*}Z\right) \leq \lambda_i\left(Z\Sigma Z^*\right) \leq \lambda_i(\Sigma)\lambda_1\left(Z^{*}Z\right).
    \end{align*}
    
    We first prove the lower bound. If $\lambda_{i+d-\min(n,d)}(\Sigma)=0$ or $\lambda_{\min(n,d)}\left(Z^{*}Z\right)=0$ the claim is trivial as $Z\Sigma Z^*$ is p.s.d. So assume they are both $>0$, meaning that $\Sigma$ has at least $i+d-\min(n,d)$ eigenvalues $\geq \lambda_{i+d-\min(n,d)}(\Sigma) >0$ and $ZZ^*$ has at least $\min(n,d)$ eigenvalues $\geq \lambda_{\min(n,d)}(ZZ^*)>0$. By \lemref{lem:variational1}, $Z\Sigma Z^*$ has at least $i$ eigenvalues that are at least $\lambda_{i+d-\min(n,d)}\left(\Sigma\right)\lambda_{\min(n,d)}\left(ZZ^*\right)$. This together with the fact that $\lambda_{\min(n,d)}\left(ZZ^*\right)=\lambda_{\min(n,d)}\left(Z^{*}Z\right)$ proves the lower bound. 

    For the upper bound, since $\Sigma$ has at least $d+1-i$ eigenvalues $\leq \lambda_{i}(\Sigma)$ and $ZZ^*$ has $n$ eigenvalues $\leq \lambda_{i}(ZZ^*)$, by \lemref{lem:variational2}, $Z\Sigma Z^*$ has at least $n+1-i$ eigenvalues that are at most $\lambda_{i}\left(\Sigma\right)\lambda_{1}\left(ZZ^*\right)$. This together with the fact that $\lambda_{1}\left(ZZ^*\right)=\lambda_{1}\left(Z^{*}Z\right)$ is equivalent to the upper bound.
\end{proof}

\subsection{Auxiliary Lemmas for \propref{prop:ostrowski}}

The following is a well-known corollary of the Courant-Fischer Min-Max theorem (see e.g. \cite{horn2012matrix}[Theorem 4.2.6])
\begin{lemma}\label{lem:min-max}
    A hermitian matrix $M\in \C^{d\times d}$ has $r$ eigenvalues that are $\geq a$ for some $a\geq 0$ if and only if there is an $r$-dimensional subspace $V\subseteq \C^d$ such that $\bv^* M\bv \geq a\bv^* \bv$ for all $\bv\in V$.
\end{lemma}

The following \lemref{lem:variational1}, \lemref{lem:variational2} and their proofs can be found in \cite{dancis1986quantitative}.
\begin{lemma}\label{lem:variational1}
Let $\Sigma \in \C^{d\times d}$ be a hermitian and $Z\in \C^{n\times d}$ be any $n \times d$ matrix. Suppose that:
\begin{enumerate}
\item $Z Z^*$ has $r$ eigenvalues that are $\geq a_1 > 0$;
\item $\Sigma$ has $s$ eigenvalues that are $\geq a_2 > 0$.
\end{enumerate}
Then the matrix $Z \Sigma Z^*$ has at least $r + s - d$ eigenvalues that are $\geq a_1a_2$.
\end{lemma}

\begin{proof}
Let $T(\bv):=Z^* \bv$ denote the linear map corresponding to $Z^*$. Since $Z Z^*$ has $r$ eigenvalues $> a_1$, by \lemref{lem:min-max} there is a subspace $V\subseteq \C^n$ with dimension $r$ such that $\bv^* ZZ^* \bv \geq a_1$ for all $\bv \in V$. Therefore $V \cap \text{Ker} Z^* = 0$, and hence the image space $T(V)$ is a linear subspace of dimension $r$.

Likewise, since $\Sigma$ has at least $s$ eigenvalues $> a_2$, there is a subspace $W\subseteq\C^d$ of dimension $s$ such that $\bw^* \Sigma \bw > a_2 \bw^* \bw$ for all $\bw\in W$.

Let $T\restriction_{V}$ denote the restriction of $T$ to the subspace $V$. $T\restriction_{V}$ is a bijective linear map from from $V$ to $T(V)$, and therefore $U := (T\restriction_{V})^{-1}\left(T(V) \cap W\right)$ is a linear subspace with $\dim U = \dim \left(T(V) \cap W \right)$. We can thus use the standard formula for the dimension of the intersection of halfspaces \citep{horn2012matrix}[Equation 0.1.7.2] to obtain
\begin{align*}
\dim \left(T(V) \cap W\right) = \dim T(V) + \dim W - \dim(T(V) + W) \geq r + s - d.
\end{align*}

Furthermore, the definition of $U$ implies for any $\bv \in U$ that $\bv\in V$ and $\bw:=Z^* \bv \in W$. Therefore, 
\begin{align*}
    \bv^* Z \Sigma Z^* \bv = \bw^* \Sigma \bw \geq a_2 \bw^* \bw = a_2 \bv^* ZZ^* \bv \geq a_1a_2 \bv^* \bv.
\end{align*}
The proof now follows from \lemref{lem:min-max}.
\end{proof}

\begin{lemma}\label{lem:variational2}
Let $\Sigma \in \C^{d\times d}$ be hermitian matrix and $Z\in \C^{n\times d}$ be any $n \times d$ matrix. Suppose that for some $b_1,b_2 \geq 0$:
\begin{enumerate}
\item $Z Z^*$ has $r$ eigenvalues that are $\leq b_1$;
\item $\Sigma$ has $s$ eigenvalues that are $\leq b_2$ for $b_2 > 0$.
\end{enumerate}
Then the matrix $Z \Sigma Z^*$ has at least $r + s - d$ eigenvalues that are $\leq b_1b_2$.
\end{lemma}

\begin{proof}
The proof is very similar to the previous lemma. Let $T(\bv):=Z^* \bv$ denote the linear map corresponding to $Z^*$. Since $-Z Z^*$ has $r$ eigenvalues $\geq - b_1$, by \lemref{lem:min-max} there is a subspace $V\subseteq \C^n$ with dimension $r$ such that for all $\bv \in V$, $\bv^* (-ZZ^*) \bv \geq -b_1$, or equivalently $\bv^* (ZZ^*) \bv \leq b_1$. However this time, as $b_1>0$, $\ker T \subseteq V$ and therefore $\dim T(V) = r - \dim\ker T$.

Again, there is a subspace $W\subseteq\C^d$ of dimension $s$ such that $\bw^* \Sigma \bw \leq b_2 \bw^* \bw$ for all $\bw\in W$.

Set $U := T^{-1}\left(T(V) \cap W\right)$ (where $T^{-1}$ is the preimage), is a linear subspace with $\dim U = \dim\ker T + \dim \left(T(V) \cap W \right)$. Furthermore,
\begin{align*}
\dim \left(T(V) \cap W\right) &= \dim T(V) + \dim W - \dim(T(V) + W) \\
& \geq r - \dim\ker T + s - d.
\end{align*}
We thus obtain that $\dim U \geq r + s -n$. Furthermore, the definition of $U$ implies for any $\bv \in U$ that $\bv\in V$ and $\bw:=Z^* \bv \in W$. Therefore, 
\begin{align*}
    \bv^* Z \Sigma Z^* \bv = \bw^* \Sigma \bw \leq b_2 \bw^* \bw = b_2 \bv^* ZZ^* \bv \leq b_1b_2 \bv^* \bv.
\end{align*}
The proof now follows from \lemref{lem:min-max}.

\end{proof}

\subsection{Proof of \thmref{thm:high_dim}} \label{ap:high_dim}
As in the rest of this paper, by \thmref{thm:gen_bound}, it suffices to bound $\norm{\frac{1}{d}Z Z^\top - I_n}_2$ in the setting of \eqref{eq:high_dim1} in order to prove it, and bound $\E\left[\norm{\frac{1}{d}Z Z^\top - I_n}_2\right]$ for \eqref{eq:high_dim2}. The following preliminary bound in \ref{lem:versh_high_dim} proves the first case. For the second case, we will build upon this lemma in \propref{prop:high_dim} that will complete the proof.

\begin{lemma}[\cite{vershynin2010introduction} Theorems 5.58, 5.62]\label{lem:versh_high_dim}
    Let $Z \in\R^{n\times d}$ be an $n\times d$ matrix for some $d\geq n\in\N$, whose rows $\bz_i$ are i.i.d. isotropic random vectors in $\R^d$ with $\norm{\bz_i}= \sqrt{d}$ a.s. 
    \begin{enumerate}
        \item For some constants $C_K,c_K>0$ which depend only on the sub-gaussian norm $K=\max_i\norm{\bz_i}_{\psi_2}$ and any $t\geq 0$ it holds w.p. at least $1-2\exp(-c_Kt^2)$ that for all $1\leq i \leq \min(n,d)$, and $\epsilon:=C_K\sqrt{\frac{n}{d}} + \frac{t}{\sqrt{d}}$,
        \begin{align*}
        \norm{\frac{1}{d}ZZ^\top - I_n}_2\leq \max(\epsilon, \epsilon^2).
        \end{align*}

        \item Letting
        \begin{align*}
        m := \frac{1}{d} \mathbb{E} \max_{j \leq n} \sum_{k \in [n], k \neq j} \langle \bz_j, \bz_k \rangle^2
        \end{align*}
        be the incoherence parameter, it holds for some constant $C>0$ and $\epsilon:=C\sqrt{\frac{m \log n}{d}}$ that
        \begin{align*}
        \E\left[\norm{\frac{1}{d}Z Z^\top - I_n}_2\right] \leq \epsilon.
        \end{align*}
    \end{enumerate}
\end{lemma}

\begin{proposition}\label{prop:high_dim}
    Let $Z\in \R^{n\times d}$ be an $n\times d$ matrix for some $d\geq n\in\N$, whose rows $\bz_i$ are independent random isotropic vectors in $\R^d$ with $\norm{\bz_i}=\sqrt{d}$ a.s. For any $p\in\N$ let $K(p) := \max_{i\in[n]}\sup_{x\in\Sphere^{d-1}}\E_{\bz_i}[\abs{\langle \bz_i, x \rangle}^{p}]^{\frac{1}{p}}$. Then, 
    \begin{align*}
    \E\left[\norm{\frac{1}{d}Z Z^\top - I_n}_2\right] \leq \epsilon.
    \end{align*}
    where
    \begin{align*}
    \epsilon := \frac{C}{\delta}\sqrt{\frac{p}{\log(p)+1}\frac{n^{\frac{1}{p}}\max\left(n, n^\frac{1}{p}K(2p)^2\right) \log(n)}{d}}
    \end{align*}
    for some absolute Constance $C>0$.
\end{proposition}
\begin{proof}
Follows the bounds in \lemref{lem:versh_high_dim}, and plugging in the bound for the incoherence parameter $m$ from \corref{cor:incoherence_bound}.
\end{proof}

\begin{lemma}[\cite{vershynin2010introduction} Lemma 5.20] \label{lem:lem:norm_iso}
Let $\bz_1, \bz_2\in\R^d$ be independent isotropic random vectors, then $\E[\norm{\bz_1}^2]=d$ and $\E[\langle \bz_1, \bz_2 \rangle^2]=d$.
\end{lemma}

\begin{lemma}\label{lem:gen_incoherence_bound}
    Let $\bz_1, \ldots, \bz_n\in \R^d$ be independent random vectors for some $n,d\in\N$. Then there exists some absolute constant $C>0$ s.t. the incoherence parameter
    \begin{align*}
    m:=\E\left[\frac{1}{d}\max_{i\leq n}\sum_{j\in[n], j\neq i}\langle \bz_i, \bz_j \rangle^2\right],
    \end{align*}
    satisfies for any $p > 1$
    \begin{align*}
    m \leq C\frac{p}{\log(p)}n^{\frac{1}{p}} \cdot \frac{1}{d}\max_{i\in[n]}\max\left(\sum_{j\in[n],j\neq i} \E\left[\langle \bz_i, \bz_j \rangle^2\right], \left(\sum_{j\in[n], j\neq i}\E\left[\langle \bz_i, \bz_j \rangle^{2p}\right]\right)^{\frac{1}{p}}\right)
    \end{align*}
\end{lemma}
\begin{proof}
    Let $D_{i,j}:=\langle \bz_i, \bz_j \rangle^2$. For any $p>1$, 
    \begin{align*}
        d\cdot m = & \E\left[\max_{i\leq n} \sum_{j\in[n], j\neq i} D_{ij}\right] 
        \leq \E\left[\max_{i\leq n} \left(\sum_{j\in[n], j\neq i} D_{ij}\right)^{p}\right]^{\frac{1}{p}} \\
        \leq & \left(\sum_{i=1}^n\E\left[\left(\sum_{j\in[n], j\neq i} D_{ij}\right)^{p}\right]\right)^{\frac{1}{p}}  
        = n^{\frac{1}{p}} \max_{i\leq n}\E\left[\left(\sum_{j\in[n], j\neq i}^n D_{ij}\right)^{p}\right]^{\frac{1}{p}}.
    \end{align*}
    For any fixed $i\in[n]$, $\{D_{ij}\}_{j\in[n],j\neq i}$ are independent and non-negative random variables, so by Rosenthal's inequality (see for example \cite{johnson1985best} or \cite{de2012decoupling}[Theorem 1.5.9]) there exists some absolute constant $C>0$ s.t. for any $i\in[n]$,
    \begin{align*}
        \E\left[\left(\sum_{j\in[n],j\neq i} D_{ij}\right)^{p}\right]^{\frac{1}{p}}
        \leq & C \frac{p}{\log(p)} \max\left(\sum_{j\in[n],j\neq i} \E\left[D_{ij}\right], \left(\sum_{j\in[n],j\neq i}\E\left[D_{ij}^{p}\right]\right)^{\frac{1}{p}}\right), 
    \end{align*}
    which completes the proof.
\end{proof}

\begin{corollary}\label{cor:incoherence_bound}
    Let $\bz_1, \ldots, \bz_n\in \R^d$ be independent and isotropic random vectors for some $n,d\in\N$. 
    For any $p\in\N$ let $K(p) := \max_{i\in[n]}\sup_{x\in\Sphere^{d-1}}\E_{\bz_i}[\abs{\langle \bz_i, x \rangle}^{p}]^{\frac{1}{p}}$.
    Then there exists some absolute constant $C>0$ s.t. the incoherence parameter
    \begin{align*}
    m:=\frac{1}{d}\E\left[\max_{i\leq n}\sum_{j\in[n], j\neq i}\langle \bz_i, \bz_j \rangle^2\right],
    \end{align*}
    satisfies
    \begin{align*}
    m \leq C\frac{p}{\log(p)+1}n^{\frac{1}{p}}\max\left(n, n^\frac{1}{p}K(2p)^2\right).
    \end{align*}
\end{corollary}
\begin{proof}
    Since $\bz_i$ is isotropic, for any $\x\in\R^d$ it holds by \lemref{lem:lem:norm_iso} that $\E_{\x}[\langle \bz_i, \x\rangle ^2] = \norm{\bz_i}^2$. In particular, it holds that for any $i\in [n]$ that
    \begin{align}\label{eq:helper_gram1}
    \sum_{j\in[n], i\neq j}\E\left[\langle \bz_i, \bz_j \rangle^2\right] = nd.
    \end{align}
    For $p=1$ the claim follows directly from this by bounding the maximum over $i\in[n]$ with the sum.
    From now on, we assume $p>1$. By the definition of $K(p)$, it holds that 
    \begin{align}\label{eq:helper_gram2}
        \E\left[\langle \bz_i, \bz_j \rangle ^{2p} \right]^{\frac{1}{p}} = &
        \E_{\bz_i}\left[\E_{\bz_j}\left[\langle \bz_i, \bz_j \rangle ^{2p} \right]\right]^{\frac{1}{p}}
        \leq \E_{\bz_i}\left[\left(K(2p)\norm{\bz_i}\right)^{2p}\right]^{\frac{1}{p}} \nonumber \\
        = & K(2p)^2 \cdot d. 
    \end{align}
    The proof now follows from plunging \eqref{eq:helper_gram1} and \eqref{eq:helper_gram2} into \lemref{lem:gen_incoherence_bound}.
\end{proof}

\vskip 0.2in
\bibliography{references}

\end{document}